\documentclass[acmtog]{acmart}

\setcopyright{rightsretained}
\copyrightyear{2021}

\acmJournal{TOG}
\acmYear{2021} \acmVolume{0} \acmNumber{00} \acmArticle{000} \acmMonth{1} \acmPrice{}\acmDOI{10.1145/3470005}

\usepackage{pdfpages}

\usepackage{centernot}
\usepackage{mathrsfs}
\usepackage{microtype}
\usepackage{multirow}
\usepackage{wrapfig}

\renewcommand{\vec}[1]{\boldsymbol{\mathrm{#1}}}

\newcommand{\N}{\mbox{\rm \hbox{I\kern-.15em\hbox{N}}}}
\newcommand{\R}{\mbox{\rm \hbox{I\kern-.15em\hbox{R}}}}

\def \N {\mbox{\rm \hbox{I\kern-.15em\hbox{N}}}}
\def \R {\mbox{\rm \hbox{I\kern-.15em\hbox{R}}}}

\newcommand{\norm}[1]{\left\Vert {#1} \right\Vert}

\newcommand{\bS}{\mathbf{S}}

\newcommand{\bc}{\mathbf{c}}

\newcommand{\bp}{\mathbf{p}}

\newcommand{\bx}{\mathbf{x}}

\newcommand{\bw}{\vec{w}}
\newcommand{\by}{\vec{y}}
\newcommand{\bZero}{\mathbf{0}}

\newcommand{\LDLT}{$LDL^T$}

\newcommand{\blambda}{\boldsymbol{\lambda}}
\newcommand{\lambdaNatural}{\blambda^{\bc = 0}}

\newcommand{\Lagr}{\mathcal{L}}

\usepackage{booktabs} % For formal tables

\usepackage[ruled]{algorithm2e} % For algorithms

\usepackage{soul} %strike through
\usepackage{bbold}

\SetAlFnt{\small}
\SetAlCapFnt{\small}
\SetAlCapNameFnt{\small}
\SetAlCapHSkip{0pt}
\IncMargin{-\parindent}

\setcitestyle{square}

\newtheorem{theorem}{Theorem}

\begin{document}

\title{SGN: Sparse Gauss-Newton for Accelerated Sensitivity Analysis}

\author{Jonas Zehnder}
\affiliation{%
  %\department{Department of Computer Science and Operations Research}
  \institution{Universit\'e de Montr\'eal}
  \department{LIGUM}
  %\streetaddress{2900 Edouard Montpetit Blvd}
  %\city{Montr\'eal}
  %\state{Qu\'ebec}
  \country{Canada}
  %\postcode{H3T 1J4}
}
\email{jonas.zehnder@umontreal.ca}

\author{Stelian Coros}
\affiliation{%
  %\department{Department of Computer Science}
  \institution{ETH Z\"urich}
  \department{CRL}
  %\streetaddress{R\"amistrasse 101}
  %\city{Z\"urich}
  \country{Switzerland}
  %\postcode{8092}
}
\email{scoros@inf.ethz.ch}

\author{Bernhard Thomaszewski}
\affiliation{%
  %\department{Department of Computer Science}
  \institution{ETH Z\"urich}
  \department{CRL}
  %\streetaddress{R\"amistrasse 101}
  %\city{Z\"urich}
  \country{Switzerland}
  %\postcode{8092}
}
\affiliation{%
  %\department{Department of Computer Science and Operations Research}
  \institution{Universit\'e de Montr\'eal}
  \department{LIGUM}
  %\streetaddress{2900 Edouard Montpetit Blvd}
  %\city{Montr\'eal}
  %\province{Qu\'ebec}
  \country{Canada}
  %\postcode{H3T 1J4}
}
\email{bthomasz@inf.ethz.ch}

\renewcommand{\shortauthors}{Zehnder, Coros and Thomaszewski}

\begin{abstract}
We present a sparse Gauss-Newton solver for accelerated sensitivity analysis with applications to a wide range of equilibrium-constrained optimization problems.
Dense Gauss-Newton solvers have shown promising convergence rates for inverse problems, but the cost of assembling and factorizing the associated matrices has so far been a major stumbling block. In this work, we show how the dense Gauss-Newton Hessian can be transformed into an equivalent sparse matrix that can be assembled and factorized much more efficiently. This leads to drastically reduced computation times for many inverse problems, which we demonstrate on a diverse set of examples. We furthermore show links between sensitivity analysis and nonlinear programming approaches based on Lagrange multipliers and prove equivalence under specific assumptions that apply for our problem setting. 
\end{abstract}



 
\maketitle

\section{Introduction}

Many design tasks in engineering involve the solution of \textit{inverse problems}, where the goal is to find design parameters for a mechanical system such that the corresponding equilibrium state is optimal with respect to given objectives.
%maybe not even talk about SQP at this point
As an alternative to conventional nonlinear programming, an approach that has recently seen increasing attention in the visual computing community is to eliminate the equilibrium constraints using sensitivity analysis. Removing redundant degrees of freedom decreases not only the problem size, it also transforms a difficult nonlinear constrained optimization problem into an unconstrained minimization problem. 

Solving such minimization problems efficiently requires derivatives of the map between parameters and state, which is given implicitly via the solution of the \textit{forward} simulation problem. While the gradient can be computed efficiently with the adjoint method, using only first-order derivative information often leads to unsatisfying convergence, even if acceleration techniques such as L-BFGS are used. Second-order sensitivity analysis promises faster convergence by requiring fewer iterations but comes at the price of dense and potentially indefinite system matrices. The second problem can be resolved by resorting to the Gauss-Newton method, which replaces the full Hessian with a positive-definite approximation. However, its dense nature greatly impedes the potential of second-order sensitivity analysis.

In this paper, we show how the dense Gauss-Newton Hessian can be transformed into an equivalent sparse matrix that can be assembled and factorized much more efficiently than its dense counterpart. 
Whereas the asymptotic complexity for dense solvers is approximately $O(n^3)$ for an $n\times n$ matrix, the cost of factorizing sparse systems depends on the sparsity pattern, which is itself problem-dependent. While meaningful asymptotic bounds for sparse factorization are hard to obtain \cite{peng2020solving}, our extensive numerical examples indicate drastically reduced computation times for a wide range of inverse design problems. We furthermore establish links between sensitivity analysis and general nonlinear programming approaches based on Lagrange multipliers and prove equivalence under specific assumptions. 

\begin{figure}[t]
	\centering
	\vspace{1em}
	\raisebox{0.57\height}{\includegraphics[width= 0.44\linewidth]{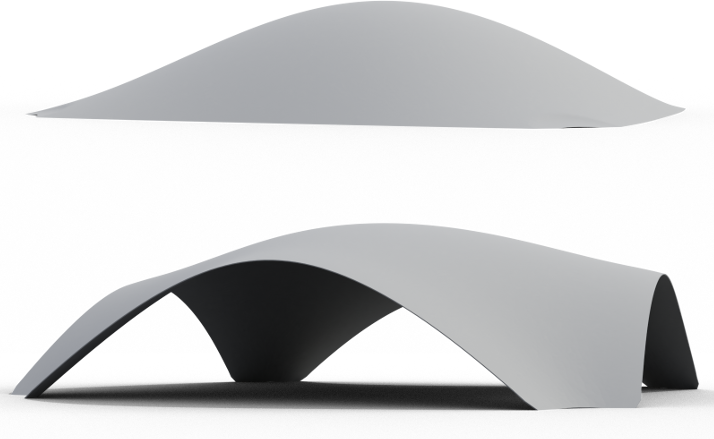}}
	\includegraphics[width= 0.55\linewidth]{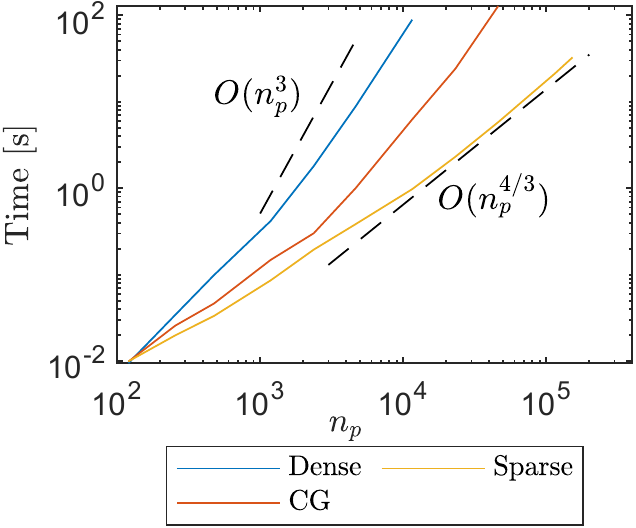}
	\caption{Comparison between dense Gauss-Newton and our sparse formulation on a shape optimization problem for a concrete shell. \textit{Left}: shell roof before (\textit{top}) and after (\textit{bottom}) optimization. \textit{Right}: average timings for computing the search direction with sparse Gauss-Newton, dense Gauss-Newton and the CG method as a function of the number of design parameters $n_p$.}
	\label{fig:teaser}
\end{figure}

\section{Related Work}

Since fabrication-oriented design moved into the focus of the visual computing community, inverse problems have received a surge of attention. While an exhaustive review of applications is beyond the scope of this work, many different methods have been proposed for solving inverse design problems ranging from musical instruments \cite{Bharaj:2015:CDM:2816795.2818108,Musialski:2016:NSO,Umetani16Printone} and balloons \cite{Skouras14DIS,skouras2012computational}, to deployable structures \cite{Guseinov17CurveUps,StringActuated:KilianMonszpartMitra:TOG:2017,Panetta:2019:XNC:3306346.3323040} and architectural-scale surfaces \cite{Vouga12SelfSupporting} and frameworks \cite{Jiang17Design,gauge2014interactive,Pietroni17PositionBased}. While some of these works propose formulations tailored to specific applications, we target general inverse problems that can be cast as constrained minimization problems with continuous parameter and state variables that are subject to equality constraints derived from physical principles.

One natural approach to solving such problems is via sequential quadratic programming (SQP), which uses states and parameters as problem variables while introducing Lagrange multipliers to enforce equilibrium constraints; see, e.g., \cite{Skouras14DIS}. As an alternative that avoids challenges associated with Lagrange multiplier formulations, the augmented Lagrangian method (ALM) has been applied to shape \cite{skouras2012computational} and multi-material \cite{skouras2013computational} optimization problems.
Another approach that has recently seen increasing attention is sensitivity analysis, which eliminates state variables and constraints such as to obtain an unconstrained minimization problem with design parameters as only variables. Sensitivity analysis is a powerful method that has been used for inverse design of mechanisms \cite{coros2013computational,Megaro:2017:CDT:3072959.3073636}, clothing \cite{Wang:2018:RSP:3197517.3201320}, material optimization \cite{Yan:2018:IDM,Zehnder:2017:MDF:3130800.3130881} as well as for optimization-based forward design \cite{umetani2011sensitive,Perez:2017:CDA}.

First-order sensitivity analysis provides the gradient of the objective function with respect to design parameters, which can be computed efficiently using the adjoint method (see, e.g., \cite{bletzinger2010optimal}). While gradient descent typically performs poorly, faster convergence for the corresponding minimization problem can be achieved using, e.g., Anderson acceleration \cite{Peng:2018:AAG:3197517.3201290}, or Quasi-Newton methods such as L-BFGS \cite{Liu:2017:QMR:3072959.2990496}. 
Kovalsky et al. \shortcite{Kovalsky:2016:AQP:2897824.2925920} improve convergence for mesh optimization by combining acceleration with Laplacian preconditioning for the gradient. Using the same preconditioner, Zhu et al. \shortcite{Zhu:2018:BCQ:3197517.3201359} instead propose a modified L-BFGS method to accelerate convergence.
These methods mostly aim at geometry deformation tasks for which the cost of evaluating the objective is relatively low. 
For physics-constrained design problems, however, the costs per step are significantly higher since each evaluation of the objective function requires simulation.

Recent work has started to investigate ways of exploiting second-order derivative information for sensitivity analysis. Panetta et al. \shortcite{Panetta:2019:XNC:3306346.3323040} use Newton's method with a trust region approach on the reduced Hessian resulting from second-order sensitivity analysis. Zimmermann et al. \shortcite{Zimmermann19Puppetmaster} propose a Generalized Gauss-Newton solver with Hessian contributions selected such as to avoid indefinite matrices. However, while this method has led to promising convergence in terms of the number of required iterations, assembling and factorizing the dense Hessian undoes this potential advantage to a large extent. This impression is further substantiated by Wang \shortcite{Wang:2018:RSP:3197517.3201320}, who benchmarked sensitivity analysis with generalized Gauss-Newton against exact and inexact gradient descent methods. 

Our method overcomes these problems through a sparse reformulation of Gauss-Newton that yields the same search direction but drastically decreases the time required for assembling and factorizing the linear system.  Although sparse Gauss-Newton formulations for sensitivity analysis have, to the best of our knowledge, not been investigated before, there are connections to so-called projected SQP methods based on reduced Hessians that have been studied in the optimization community \cite{Heinkenschloss96ProjectedSQP}. We use these insights to show equivalence between sensitivity analysis and nonlinear programming with Lagrange multipliers for specific equilibrium-constrained optimization problems.

\section{Background}

We consider constrained optimization problems of the form
\begin{equation}
\label{eq:optProb}
\min_{\bx, \bp} f(\bx, \bp) \quad \text{s.t.} \quad \bc(\bx, \bp) = \mathbf{0}\ ,
\end{equation}
where $\bx$ denotes the equilibrium state of a mechanical system described by a set of design parameters $\bp$. The state $\bx$ is coupled to the design parameters $\bp$ through a set of constraints $\bc$ requiring that $\bx$ must be an equilibrium configuration for $\bp$. While the exact form of these equilibrium constraints depends on the problem, we will focus on static and dynamic force balance in this work.

\subsection{Sensitivity Analysis}
We exclusively consider the special but common case of problems which exhibit exactly as many equality constraints as state variables, i.e., $n_\bc=n_\bx$.
Furthermore, we assume that the constraint Jacobian $\frac{\partial \bc}{\partial \bx}$ has full rank.
Under these conditions, the implicit function theorem asserts that any choice of $\bp$ in a local neighborhood uniquely determines the corresponding equilibrium state $\bx$ and we therefore write $\bx=\bx(\bp)$.
Given a state-parameter pair $(\bx,\bp)$ that satisfies the equilibrium constraints, we require that any change to the design parameters induces a corresponding change in state such that the system is again at equilibrium. Formally, we have
\begin{equation}
\frac{d\bc}{d\bp}
= \frac{\partial \bc}{\partial \bp} + \frac{\partial \bc}{\partial \bx}\frac{d \bx}{d \bp}
= \bZero \ ,
\label{eq:dcdp}
\end{equation}
from which we directly obtain the so called \textit{sensitivity matrix}
\begin{equation}
\label{eq:dxdp}
\bS=\frac{d \bx}{d \bp} = - \left( \frac{\partial \bc}{\partial \bx} \right)^{-1} \frac{\partial \bc}{\partial \bp} \ .
\end{equation}
Using this relation, the gradient of the objective function with respect to the parameters is obtained as
\begin{equation}
\label{eq:totalGradient}
\frac{d f(\bx(\bp),\bp)}{d\bp} 
= \frac{\partial f}{\partial \bp} + \frac{\partial f}{\partial \bx} \bS  \ .
\end{equation}
We note that, by using (\ref{eq:dcdp}) in the above expression and rearranging terms, computing the gradient requires only the solution of a single linear system.

\subsection{Gauss-Newton}
With the gradient defined through (\ref{eq:totalGradient}), we can minimize $f$ using steepest descent in parameter space. Every step amounts to updating $\bp$ along the search direction, computing an equilibrium configuration $\bx$ through simulation, and evaluating the objective to accept or reject the step. Though simple, the convergence of steepest descent is typically very slow. Using the Hessian of the objective function, Newton's method promises quadratic convergence close to the solution. However, Newton's method is often plagued by indefiniteness on the road towards the optimum, requiring expensive regularization and other advanced strategies. 
As a promising middle ground, Gauss-Newton retains parts of the Hessian information but is guaranteed to never encounter indefiniteness.

Gauss-Newton in its original form is a minimization algorithm for objective functions in nonlinear least-squares form,
\begin{equation}
f(\bx, \bp)=\sum_i \frac{w_i}{2} r_i(\bx, \bp)^2  \ ,
\label{eq:leastSquaresObjective}
\end{equation}
where $\bw=(w_1, \ldots, w_n)$ is a vector of weights and $r_i$ are residuals. 
Instead of using the full Hessian 
\begin{equation}
H = \sum_i w_i\frac{dr_i}{d \bp}^T\frac{dr_i}{d \bp} + \sum_iw_ir_i\frac{d^2 r_i}{d\bp^2} \ ,
\end{equation}
Gauss-Newton drops the second term to define an approximate but positive-definite Hessian. 
Writing out $d r_i / d \bp$ we arrive at 
\begin{equation}
H_{GN} = 
\begin{bmatrix}
\frac{d \bx}{d \bp}^T & I
\end{bmatrix} 
\left( \sum_i w_i
\begin{bmatrix}
\frac{\partial r_i}{\partial \bx}^T \frac{\partial r_i}{\partial \bx} &  \frac{\partial r_i}{\partial \bx}^T \frac{\partial r_i}{\partial \bp} \\
\frac{\partial r_i}{\partial \bp}^T \frac{\partial r_i}{\partial \bx} & \frac{\partial r_i}{\partial \bp}^T \frac{\partial r_i}{\partial \bp} 
\end{bmatrix}
\right)
\begin{bmatrix}
\frac{d \bx
}{d \bp} \\
I
\end{bmatrix}
\ .
\label{eq:HGN}
\end{equation}
A Gauss-Newton step can then be computed by solving the system of linear equations
\begin{equation}
\label{eq:newtonStep}
H_{GN} \cdot \delta \bp = - \frac{d f}{d \bp}^T \ .
\end{equation}
In practice, however, computing, assembling, and factorizing this reduced Hessian matrix is exceedingly expensive: it requires the complete sensitivity matrix, products between sparse matrices with incompatible sparsity patterns, and leads to a dense matrix that is expensive to factorize. 

\section{Sparse Gauss-Newton}
\label{sec:SGN}
To arrive at a more efficient formulation, we start by rewriting the Gauss-Newton Hessian as
\begin{equation}
\label{eq:reducedGN}
H_{GN} = 
\frac{d \bx}{d \bp}^T A \frac{d \bx}{d \bp}  + B\frac{d \bx}{d \bp} + \frac{d \bx}{d \bp}^T B^T
+C
\end{equation}
with
\begin{equation*}
 A=\sum_i w_i \frac{\partial r_i}{\partial \bx}^T \frac{\partial r_i}{\partial \bx} \ , \; B=\sum_i w_i \frac{\partial r_i}{\partial \bp}^T \frac{\partial r_i}{\partial \bx}\ , \; \text{and } \; C= \sum_i w_i\frac{\partial r_i}{\partial \bp}^T \frac{\partial r_i}{\partial \bp}.
\end{equation*}
Since $\frac{d \bx}{d \bp} = - \left[ \frac{\partial \bc}{\partial \bx} \right]^{-1} \frac{\partial \bc}{\partial \bp}$, the inverse of the constraint Jacobian appears inside the definition of $H_{GN}$. We can remove this inverse by reformulating the problem with additional variables $\delta \bx = \frac{d \bx}{d \bp} \delta \bp$, which is equivalent to $\frac{\partial \bc}{\partial \bp} \delta \bp + \frac{\partial \bc}{\partial \bx} \delta \bx = 0 $,
\begin{equation}
\label{eq:firstSystemGN}
\begin{bmatrix}
\frac{d \bx^T}{d \bp} A + B  &
C + \frac{d \bx^T}{d \bp} B^T  \\
\frac{\partial \bc}{\partial \bx} &
\frac{\partial \bc}{\partial \bp} 
\end{bmatrix}
\begin{bmatrix}
\delta \bx \\
\delta \bp
\end{bmatrix}
= 
\begin{bmatrix}
- \frac{d f}{d \bp}^T \\
0
\end{bmatrix} \ .
\end{equation}
However, the transpose of the sensitivity matrix still appears in this system. To also remove this occurrence, we introduce additional variables $\delta \blambda$ defined as
\begin{equation}
\label{eq:GNdlambda}
\frac{\partial \bc}{\partial \bx}^T \delta \blambda + 
B^T \delta \bp + 
A \delta \bx= 0 \ ,
\end{equation}
and arrive at the extended system
\begin{equation}
\label{eq:sparseSystemGN}
\begin{bmatrix}
A&
B^T &
\frac{\partial \bc}{\partial \bx}^T\\
B &
C &
\frac{\partial \bc}{\partial \bp} ^T\\
\frac{\partial \bc}{\partial \bx} &
\frac{\partial \bc}{\partial \bp} &
0
\end{bmatrix}
\begin{bmatrix}
\delta \bx \\
\delta \bp \\
\delta \blambda
\end{bmatrix}
= 
\begin{bmatrix}
0 \\
- \frac{d f}{d \bp}^T\\
0
\end{bmatrix} \ .
\end{equation}
Note that the first row enforces (\ref{eq:GNdlambda}), whereas the second row is obtained by using (\ref{eq:GNdlambda}) and (\ref{eq:dxdp}) in the first row of (\ref{eq:firstSystemGN}).
The resulting system is sparse and it requires neither the inverse of the constraint Jacobian, nor the sensitivity matrix.
We emphasize that, by construction, the solution $\delta\bp$ obtained when solving this system is \textit{exactly identical} to the one obtained when factorizing the dense Hessian.
Although this new system is larger than the reduced one, its sparsity allows us to leverage specialized linear solvers. The exact time complexity of common sparse direct solvers is only known for specific sparsity patterns and can range from $O(n)$ for, e.g., a diagonal matrix to $O(n^3)$ for a quasi-dense matrix \cite{peng2020solving}. Nevertheless, our experiments show that the cost of factorizing the larger sparse system is \textit{asymptotically lower} than the cost of factorizing the reduced dense system for many problems; see Fig. \ref{fig:teaser} for an example. This result translates into dramatically improved performance for a large range of problems, as we demonstrate with our examples.

\subsection{Discussion and Generalization}
\label{sec:GeneralizationGN}

\paragraph{Relation to Sequential Quadratic Programming}
System (\ref{eq:sparseSystemGN}) is in the form of a saddle-point problem that is characteristic for first-order optimality conditions in nonlinear programming. 
Indeed, it can be shown that second-order sensitivity analysis on general objectives is equivalent to so called reduced SQP methods when using a particular definition for the Lagrange multipliers; see \cite{Reyes15Numerical} and our derivations in Appendix \ref{sec:app}.

\paragraph{Generalization to Arbitrary Objectives}
Our construction can be extended to general objectives $f(\bx(\bp),\bp)$ for which the Hessian reads
\begin{equation}
\begin{split}
\frac{d^2 f}{d \bp^2} = &\frac{d \bx}{d \bp}^T \frac{\partial^2 f}{\partial \bx^2} \frac{d \bx}{d \bp} + \frac{\partial^2 f}{\partial \bx \partial \bp} \frac{d \bx}{d \bp} + \frac{d \bx}{d \bp}^T \frac{\partial^2 f}{\partial \bp \partial \bx} 
 + \frac{\partial^2 f}{\partial \bp^2} \\
& + \sum_i \frac{\partial f}{\partial \bx_i} \frac{d^2 \bx_i}{d \bp^2} \ .
\end{split}
\label{eq:totalHessian}
\end{equation}
In particular, when dropping only second-order sensitivities to obtain the Generalized Gauss-Newton approximation \cite{Zimmermann19Puppetmaster}, our formulation applies directly with blocks defined as $A=\frac{\partial^2 f}{\partial \bx^2}$, $B=\frac{\partial^2 f}{\partial \bx \partial \bp}$, and $C=\frac{\partial^2 f}{\partial \bp^2}$. The extension to the full-Hessian case and its relation to nonlinear programming is described in Appendix \ref{sec:app}.
It should be noted, however, that neither Generalized Gauss-Newton nor full Newton offer any guarantees on the positive-definiteness of the blocks. In our experiments, the additional measures required for detecting and treating indefiniteness can easily undo the potential advantage of using more accurate Hessian information.

\paragraph{Combination with L-BFGS}
As we show in Sec. \ref{sec:results}, Gauss-Newton leads to very good convergence in many cases and our sparse formulation makes this approach highly efficient. Nevertheless, Gauss-Newton is not a true second-order method and, depending on the problem, the missing derivative information can slow down convergence. 
For such cases, combining Sparse Gauss-Newton with L-BFGS can be an attractive alternative: even though L-BFGS requires only first-order derivatives, it approximates second-order information in its inverse Hessian from rank-one updates with past gradients. Similar in spirit to \cite{Kovalsky:2016:AQP:2897824.2925920,Liu:2017:QMR:3072959.2990496,Zhu:2018:BCQ:3197517.3201359}, we use Sparse Gauss-Newton to initialize the inverse Hessian approximation in L-BFGS. This amounts to solving a linear system each time a new search direction is computed. We provide an evaluation of this approach in Sec. \ref{sec:results}.

\paragraph{Block Solve}
If the objective $f$ does not directly depend on the design parameters, the block structure for (\ref{eq:sparseSystemGN}) simplifies to $B=0$ and $C=0$. If the constraint Jacobian $\partial \bc/\partial \bp$ is invertible, upon block substitution, the solution $\delta \bp$ is obtained by solving the linear system
\begin{equation}
	\frac{\partial \bc}{\partial \bp} \delta \bp =  -\frac{\partial \bc}{\partial \bx} \delta \by \ , \quad \text{with} \quad \delta \by= A^{-1}\frac{\partial f}{\partial x}^T \ . 
\end{equation} 
See Appendix \ref{app:blockSolve} for a detailed derivation. If the objective is not a simple $L_2$ distance, then $A\neq I$ and we must solve an additional linear system to obtain $\delta \by$.
We show in Sec. \ref{sec:results} that, where applicable, this block solve can accelerate the computation of the search direction by another $30\%$ and more compared to the Sparse Gauss-Newton baseline.

\section{Results}
\label{sec:results}

We evaluate the performance of our Sparse Gauss-Newton (SGN) solver on a set of inverse design problems. Besides illustrating different applications, each of these problems differs in terms of the ratio between parameters and state variables, the connectivity among variables, as well as their degree of nonlinearity and convexity. 
We are primarily interested in assessing the relative performance of SGN and dense Gauss-Newton (DGN), and how this ratio evolves as a function of problem size. Since both methods give the same results, we only provide average computation times for computing search directions in most cases. Additionally, we also provide total computation times on selected examples and compare to alternative approaches. We measure convergence in terms of suboptimality, which we define as the objective function value minus its value at the minimum.

\paragraph{Solving the saddle point problem}
 We used the PARDISO LU direct solver from Intel's Math Kernel Library (MKL) for solving the indefinite sparse linear systems. This solver performed robustly and efficiently for all problem types and resolutions except for the cloth example, where it returned solutions of insufficient accuracy. Instead of tweaking solver parameters per problem, we opted for a robust fall-back strategy based on the iterative BiCGSTAB method~\cite{van1992bi}, using PARDISO's \LDLT decomposition of the stabilized matrix as preconditioner.
 Specifically, we add the vector $[
    \varepsilon_x \mathbb{1}_{n_x}^T,    \mathbb{0}_{n_p}^T  ,
    - \varepsilon_{\lambda} \mathbb{1}_{n_c}^T
    ]^T$
to the diagonal of the matrix with
$\varepsilon_x = 10^{-6}$ and $\varepsilon_{\lambda} = 10^{-6}$.
We found this strategy to work well in practice, requiring only a few BiCGSTAB iterations to solve the system to high accuracy. It should be noted that the optimal stabilization of the lower right block is scale dependent and should be chosen according to the norm of the constraint Jacobian.  See~\cite{benzi2005numerical} for more details on stabilization, and on the numerical solution of saddle point problems in general.

 We also experimented with iterative solvers such as BiCGSTAB and GMRES with a variety of commonly used preconditioners, ranging from simple diagonal scaling (Jacobi)  to incomplete factorization (ILUT) methods. For the problems considered in this work, however, these iterative methods were either much slower or failed to converge at all. Though we expect iterative solvers to eventually outperform direct solvers for increasingly large problems sizes, we consider this topic beyond the scope of this work. For Sparse Newton and Sparse GGN we used the inertia revealing feature of PARDISO's \LDLT decomposition to test for positive-definiteness on the nullspace of the constraint Jacobian (i.e., second-order optimality conditions) and added diagonal regularization if necessary~\cite{han1985inertia}. For the trust region method we used \textit{trlib}~\cite{lenders2018trlib} to solve the trust region subproblem. We solve the reduced linear systems (DGN) using Eigen's built-in Cholesky decomposition.
 All measurements that we present here were done on an Intel i7-6700K quad-core with 16GB of RAM.

\paragraph{Computing Equilibrium States}
All methods based on sensitivity analysis must recompute the equilibrium state through forward simulation after each parameter update. Forward simulation amounts to a nonlinear minimization problem, whose objective function depends on the application. In each case, we ensure monotonicity in the objective using a backtracking line search.
We use standard computational models described in the literature. Derivatives with respect to state and design parameters are computed analytically using pre-compile-time automatic differentiation.

\begin{figure}
	\centering
	\includegraphics[width=0.9\linewidth]{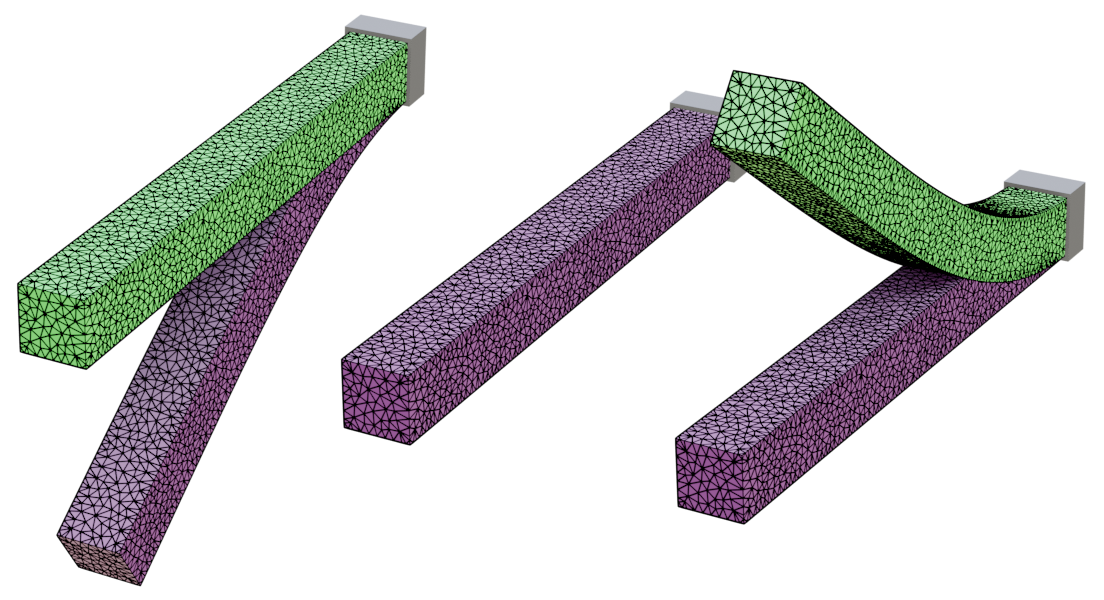}	
	\includegraphics[width= 0.49\linewidth]{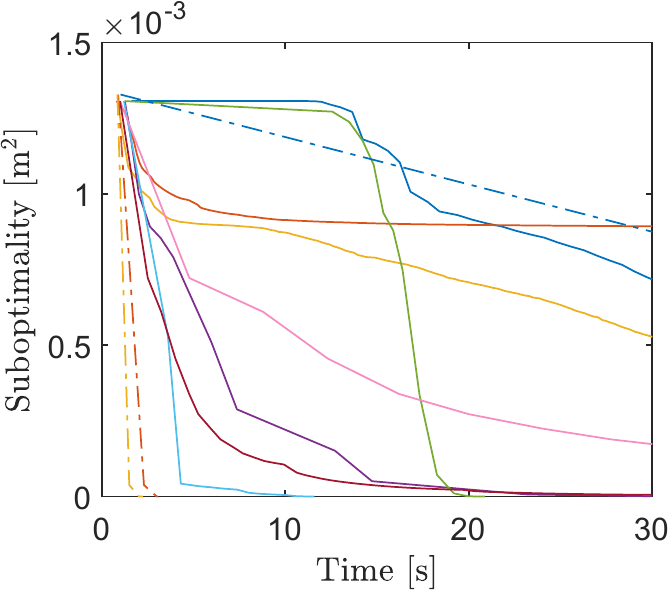}
	\includegraphics[width= 0.49\linewidth]{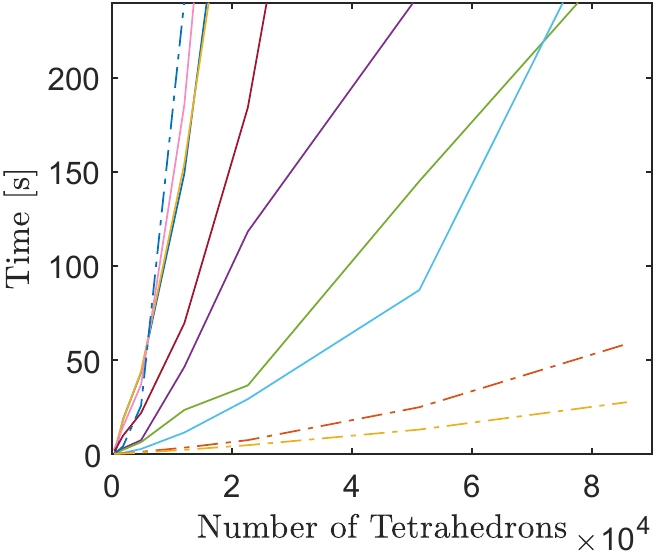}
	\includegraphics[width=0.93\linewidth]{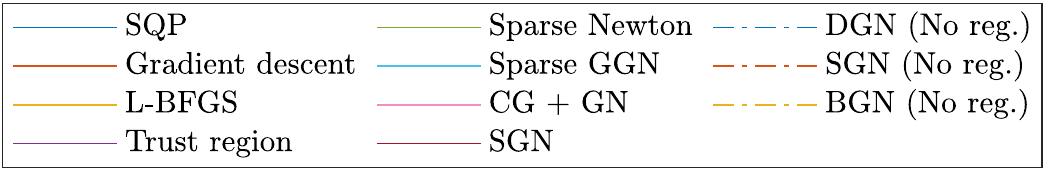}
	\caption{Performance comparison for different solvers on an inverse elastic design problem. \textit{Top left}: initial rest shape (\textit{green}) and corresponding deformed state (\textit{purple}). \textit{Top middle}: target shape. \textit{Top right}:  optimized rest shape (\textit{green}) and corresponding deformed state (\textit{purple}). \textit{Bottom left}: objective value vs. computation time for a mesh size of 3228 vertices. \textit{Bottom right}: computation time vs. problem size.
	}
	\label{fig:solidBar}
\end{figure}

\begin{figure*}[ht]
	\centering
	\includegraphics[scale=0.65]{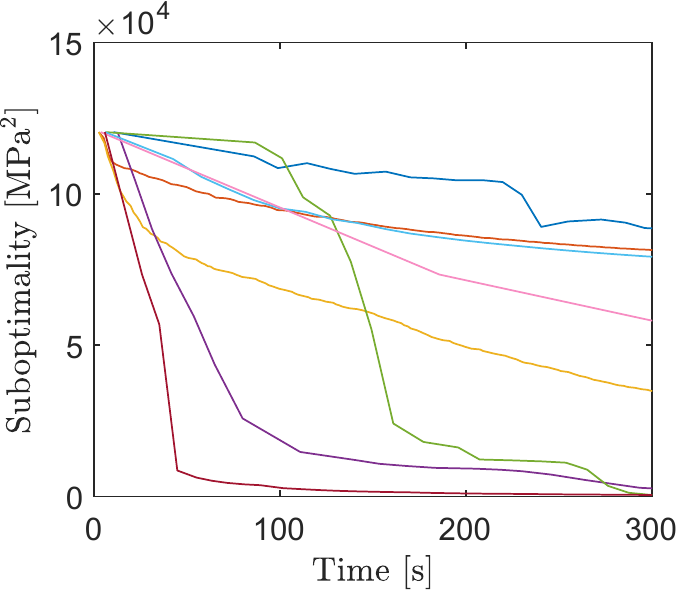}
	\includegraphics[scale=0.65]{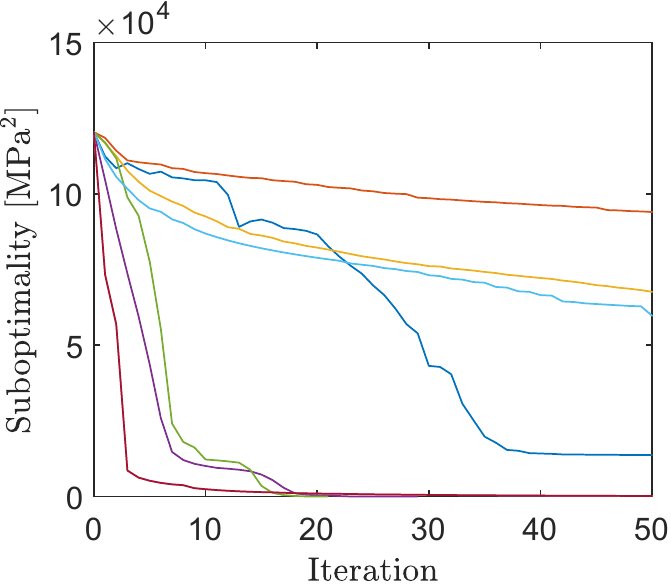}
	\includegraphics[scale=0.65]{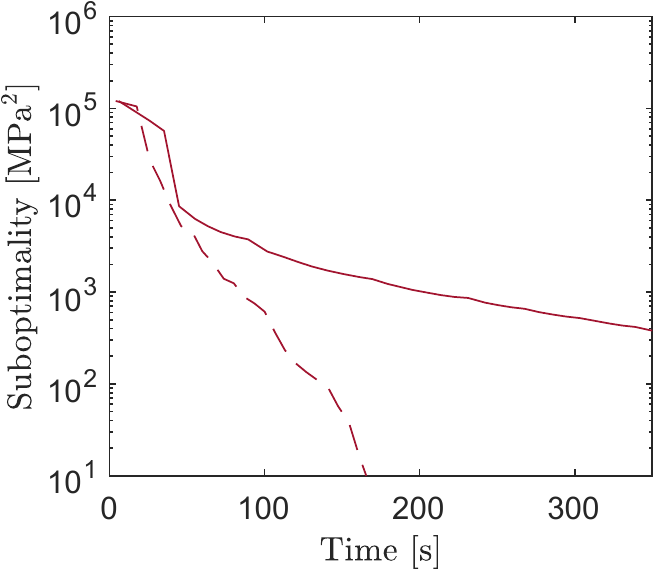}
	\raisebox{5.9mm}[0pt][0pt]{
	\hspace{2mm}
	\includegraphics[scale=0.077]{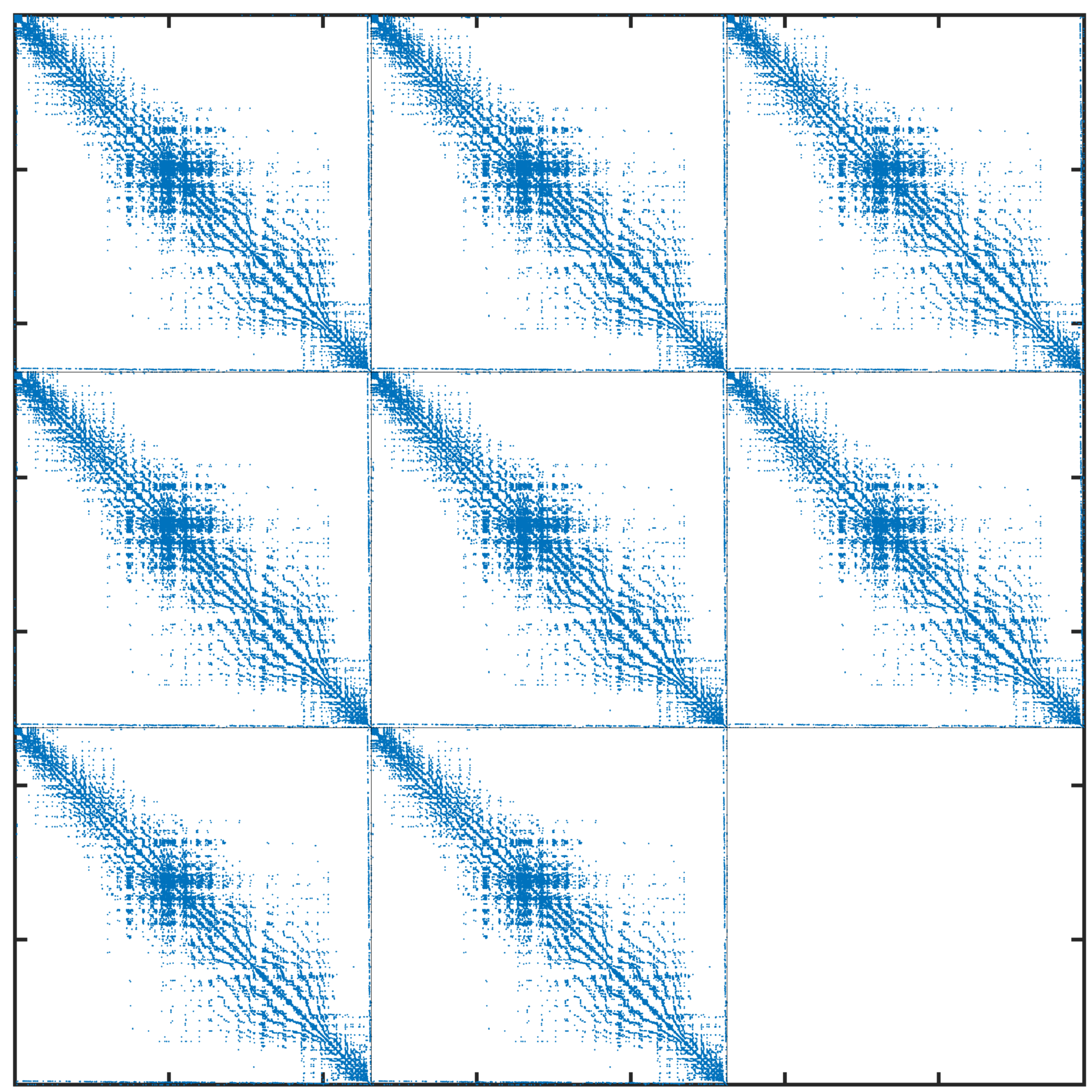}
	}
	\includegraphics[scale=0.65]{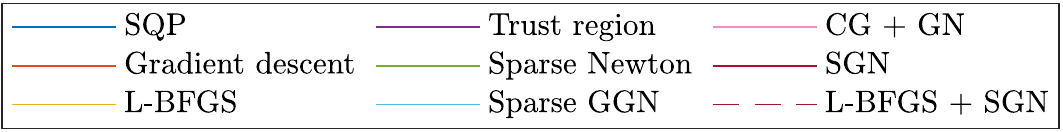}
	\caption{Convergence of different solvers for the shell roof example with 15,443 vertices. Dense Gauss-Newton is not listed since the linear solver ran out of memory when computing the reduced Hessian. 
	The sparsity pattern of the saddle point system, shown on the right, reveals a repetitive structure resulting from similar stencils for objective and constraints. The dense hessian contains $2.13\cdot10^9$ entries whereas the sparse KKT matrix contains $1.77\cdot10^7$ nonzero entries. PARDISO reported $2.18\cdot 10^8$ nonzero entries in the decomposition.
 }
	\label{fig:shellRoof}
\end{figure*}

\subsection{Inverse Elastic Design}
Our first example considers gravity compensation for a simple elastic bar clamped on one side and subjected to gravity; see Fig. \ref{fig:solidBar}. The goal is to find a rest state mesh such that the resulting equilibrium state is as close as possible to a given target shape,
\begin{align}
	f(\bx,\bp) = \frac1{2n_x} \norm{\bx - \bx_\text{target}}^2 + R(\bp) \ .
\end{align}
To prevent inversions in the rest shape $\bp$, we add a nonlinear least-squares regularizer $R(\bp)$ that penalizes per-element volume changes. For the forward simulation,  we use standard linear tetrahedron elements and a Neo-Hookean material with Young's modulus $E=10^6Pa$ and Poisson's ratio $\nu=0.45$. As termination criterion we used $\norm{d f / d\bp} \leq 10^{-5}$ where the gradient norm at the beginning of the optimization was close to $2\cdot 10^{-2}$ for all resolutions.

While there are specialized solvers for gravity compensation problems \cite{Mukherjee:2018:ITS:3242771.3203196,Chen:2014:ANM,Ly2018inverse}, we use this example as a benchmark for evaluating the relative performance of SGN, DGN, as well as sparse versions of full Newton and Generalized Gauss-Newton (GGN). For comparison, we also add alternative approaches based on sensitivity analysis that have recently been introduced or used in the visual computing community: the trust region solver by Panetta et al. \shortcite{Panetta:2019:XNC:3306346.3323040}, as well as standard Gradient Descent and L-BFGS. We also include performance data for our implementation of Sequential Quadratic Programming (SQP) using Newton's method on the KKT-conditions (see Appendix \ref{sec:app}) and an exact $L_1$ merit function.
 As a further reference point, we also compare to the conjugate gradient method (CG) applied directly to Equation~{\ref{eq:reducedGN}} and use back-substitutions to avoid forming the dense matrix. We refer to this alternative as \textit{CG + GN}. As termination criterion for CG, we use a relative residual threshold $\eta $ of $10^{-3}$ 
  for all examples.
We furthermore note that, when dropping the regularizer, the shape objective does not directly involve the design parameters, allowing us to apply the block Gauss-Newton (BGN) method described in Sec. \ref{sec:GeneralizationGN}. Interestingly, we found that the Gauss-Newton methods produce smooth rest state deformations even without regularizer, whereas all other methods led to inversions in that case. In the case of CG + GN, the default value of the threshold $\eta$ was not low enough to avoid inversions and thus this method was not included.

From Fig. \ref{fig:solidBar} (\textit{left}) it can be seen that sparse methods without regularizer, i.e., our SGN solver and its block-solve version (BGN, outperform all other solvers by a relatively large margin.
Sparse GGN and the trust region solver rank first among the alternative approaches, keeping track with SGN when using regularization. Gradient descent and L-BFGS initially perform well but progress quickly slows down. The remaining methods are not competitive on this example. The superior scaling behavior of BGN and SGN without regularizer becomes evident from Fig. \ref{fig:solidBar} (\textit{right}). Although the difference between BGN and SGN is small compared to the other methods, the block-solve version still offers about $30\%$ performance increase for lower resolutions and more than $50\%$ for higher resolutions.

\subsection{Shell Form Finding}
For the solid bar example, the objective was a simple $L_2$-distance on the equilibrium state which, even with regularization, did not directly couple parameters and state.
In our second example, we investigate a case in which these quantities are strongly coupled---a form finding problem for a $50m\times50m$ concrete shell roof, inspired by the works of architect F{\'e}lix Candela; see \cite{Tomas10Optimality} and Fig. \ref{fig:teaser}. We model the roof using discrete shells~\cite{Grinspun:2003:DS:846276.846284} with material parameters corresponding to $E=28GPa$, $\nu=0.2$ as well as a density of $2500kg/m^3$.
The design task consists in finding a rest shape for the shell such that the equilibrium state under gravity minimizes a stress objective in nonlinear least squares form. We use a stress model following~\cite{gingold2004discrete}. To encourage smooth solutions, we additionally use a regularizer $R(\bp)$ that penalizes curvature in the rest state and per-triangle deformations. The resulting objective is
\begin{align}
    f(\bx, \bp) = \sum_i ||\sigma_i(\bx, \bp)||^2 + R(\bp) \ ,
\end{align}
where $\sigma_i$ denotes the Cauchy stress for element $i$.
We note that this objective couples $\bx$ and $\bp$, meaning that we cannot apply BGN, and its Hessian is not guaranteed to be positive-definite.
As can be seen from Fig. \ref{fig:shellRoof} (\textit{left}), SGN clearly outperforms all other methods. Interestingly, the sparse Generalized Gauss-Newton performs similar to Gradient Descent and far worse than L-BFGS, which can be attributed to the Hessian approximation becoming indefinite.

Although SGN rapidly decreases the initial objective by almost two orders of magnitude, the log-scale plot in Fig. \ref{fig:shellRoof} (\textit{right}) reveals that convergence slows down afterwards. However, the combination of SGN with L-BFGS as described in Sec. \ref{sec:SGN} is able to sustain rapid convergence in this case.  

For this example, dense Gauss-Newton ran out of memory when trying to compute the dense reduced Hessian. Fig. \ref{fig:teaser} nevertheless compares timings between SGN and DGN for smaller problem sizes, indicating that, due to its different asymptotic complexity, SGN breaks even already for problem sizes beyond a few hundred parameters.

We note that, in order to allow for larger geometry changes, we used a lower regularization weight for the result shown in Fig. \ref{fig:teaser} than for the comparison shown in Fig. \ref{fig:shellRoof} since otherwise, not all methods would converge to the same solution.

\begin{figure}
	\centering
	\raisebox{0.3\height}{\includegraphics[width=0.44\linewidth]{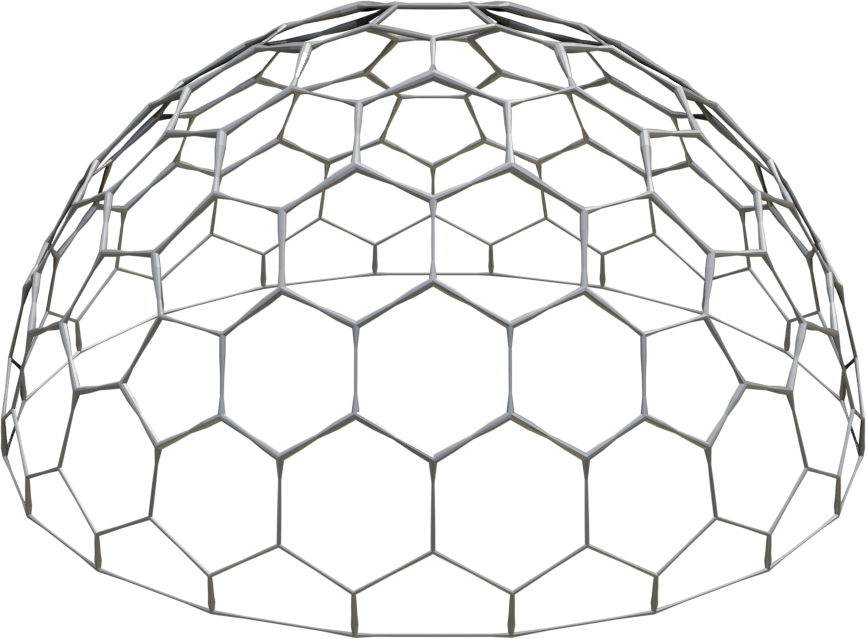}}
	\includegraphics[width=0.55\linewidth]{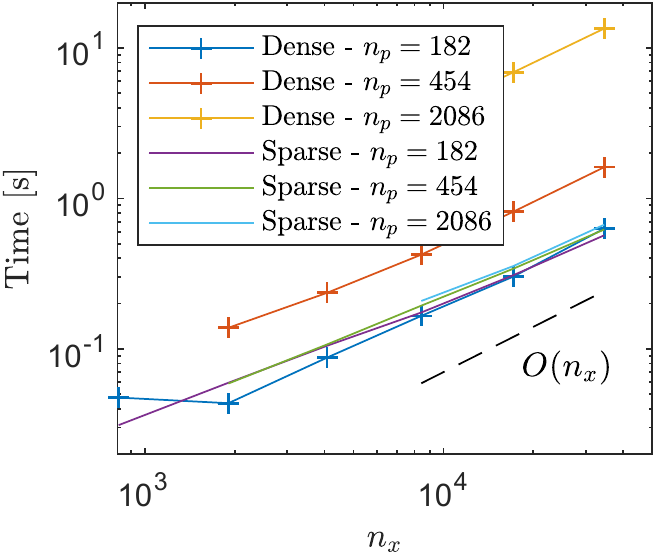}
	\caption{Time to compute the search direction for the rod dome as a function of state size $n_x$. We subdivide the edges to add more state variables and compare timings for different numbers of parameters.}
	\label{fig:rodDome}
\end{figure}

\subsection{Rod Dome}
In the third example, we consider an inverse design problem for a hemispherical dome made from interconnected elastic rods. 
The dome is subject to a force applied at the top and we impose Dirichlet boundary conditions on the bottom. The design parameters are radii for the rods that are prescribed at connection points and interpolated along the rods. The goal, then, is to find parameters that minimize a weighted combination of the total mass of the structure---approximated as the $L_2$ norm of the parameter vector---and its displacement under load. We define the design objective in nonlinear least-squares form as
\begin{align}
	f(\bx, \bp) = \frac12 \norm{ \bx - \bx_\text{undef} }^2 + \frac12 \norm{ \bp }^2_V +  f_\text{bounds}(\bp) \ ,
\end{align}
where $f_\text{bounds}(\bp)$ is a log-barrier term enforcing lower and upper limits on the radii. The mapping $\norm{\cdot}_V^2$ approximates the volume of the structure using conical frustums.
As simulation model, we used discrete elastic rods \cite{Bergou:2010:DVT:1778765.1778853} together with the extension to rod networks by Zehnder et al. \shortcite{Zehnder16DSO} and set material parameters to $E=69GPa$ and $\nu=0.33$ such as to emulate aluminum rods.

The problem setup as described above allows us to independently vary the number of parameters and state variables. As can be seen from Fig. \ref{fig:rodDome}, for small numbers of parameters, DGN outperforms SGN. However, for a given number of state variables, SGN shows only a slight growth in computation time when the number of parameters is increased. The situation is very different for DGN and the break even point for this example is at around 200 parameters. It is worth noting that both dense and sparse Gauss-Newton show a similar increase in computation time with increasing number of state variables. For SGN, we conjecture that the linear scaling is due to the particular sparsity structure induced by the rod dome, which---except for the connecting nodes---exhibits a band-diagonal structure. 

\subsection{Car Control}
The examples studied so far investigated the performance and scalability of our method for static equilibrium problems. We now turn to inverse dynamics problems in which we seek to optimize for control parameters such that the resulting dynamic equilibrium motion optimizes given design goals.
For the first of two examples, we consider the problem of steering a simple self-driving car such as to move from a given starting position to a prescribed goal configuration. 
The state of the car $\bx=(p_x,p_y,\theta)$ is described by three variables representing its position $(p_x, p_y)$ on the plane and angle $\theta$ with respect to the first coordinate axis. The parameters $\bp$ are control variables that include the speed $v$ in forward direction and the steering angle $s$ relative to the forward direction.
The motion of the car is described by the simple first-order ODE $\dot{\bx} =(v \cos \theta, v \sin \theta, v \tan s)$, which we formulate in constraint form as 
\begin{align}
	\bc^{t_i}(\bx, \bp) = \bx^{t_i} - I_{EE}(\bx^{t_{i-1}}, \bp^{t_i}) \ ,
\end{align}
where the time integration routine $I_{EE}$ takes state variables $\bx^{t_{i-1}}$ and control variables $\bp^{t_i}$ at the beginning of a given time step and returns the corresponding new state $\bx^{t_i}$. We use explicit Euler integration with a step size of $\frac1{30}s$ and run the simulation for $N$ steps. The total number of state and problem variables is therefore $3N$ and $2N$, respectively.
The objective that we minimize measures the difference between final and target states as 
\begin{align}
	f(\bx, \bp) = \frac{w_\text{pos}}{2}||\bx^N-\bx_\mathrm{target}||^2  + \frac{w_\text{dir}}{2}||\mathbf{d}(\bx^N)-\mathbf{d}_\mathrm{target}||^2 + R(\bp)\ ,
\end{align}
where $\mathbf{d}$ is the vector that points in the forward direction of the car and $R(\bp)$ is a smoothness term that penalizes differences in control variables over time. Except for L-BFGS-B, we additionally enforce bounds on the maximum velocity and steering angle by filtering the search direction. 

\begin{figure}
	\centering
	\raisebox{0.22\height}{\includegraphics[width=0.49\linewidth]{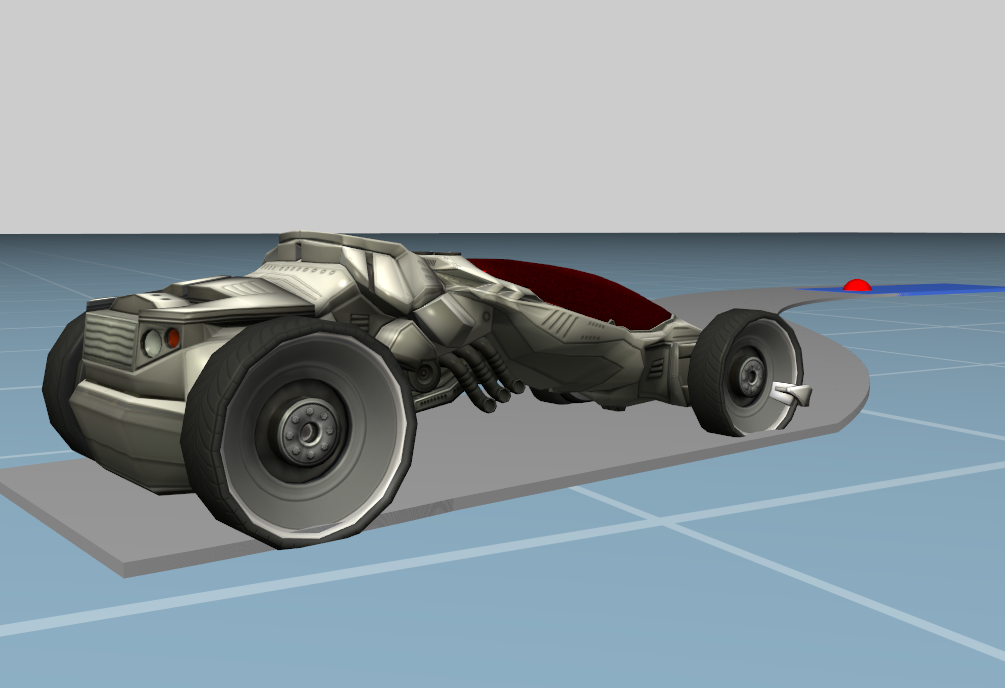}}
	\includegraphics[width=0.49\linewidth]{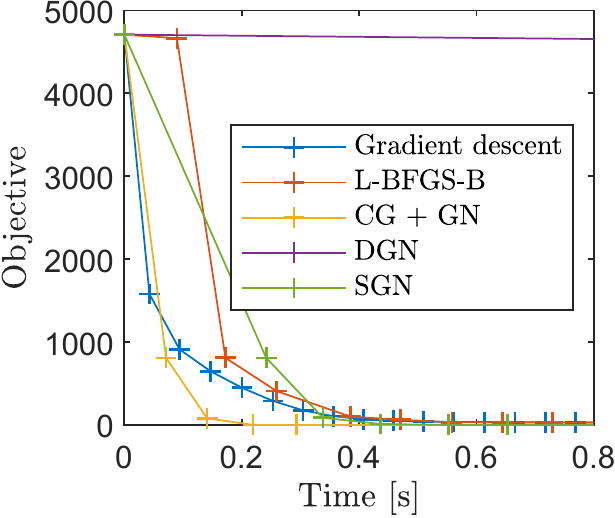}
	\caption{Performance comparison for different solvers on the car example using 5000 time steps.}
	\label{fig:car}
\end{figure}
\begin{figure}
	\centering
	\includegraphics[width=0.50\linewidth]{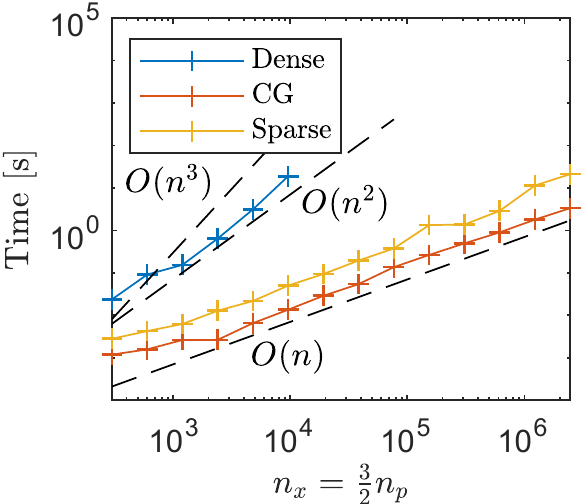}
	\raisebox{0.01\height}{
	\includegraphics[width=0.45\linewidth]{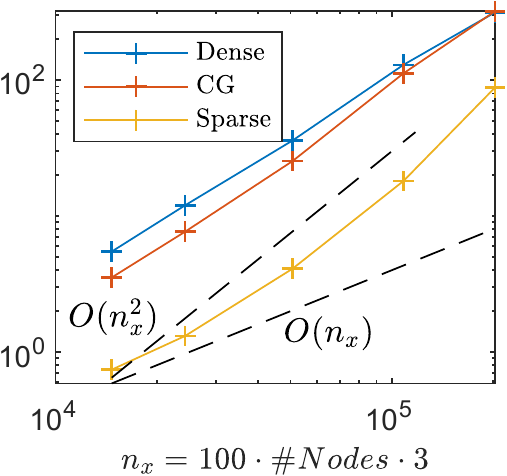}}
	\caption{Average time for computing Gauss-Newton search directions for the car (\textit{left}) and cloth (\textit{right}) control examples using the dense vs. our sparse Hessian and the CG method. \textit{Left}: we increase the number of time steps, thus increasing problem size in terms of both state variable and parameters.
	\textit{Right}: we increase the number of vertices for the cloth while keeping the number of time steps fixed and thus the number of control variables constant.
	}
	\label{fig:scalingCarCloth}
\end{figure}

For this relatively simple and non-stiff problem, Gradient Descent performs comparatively well and is only slightly slower than Sparse Gauss-Newton. While CG + GN outperforms all other methods in this example, the difference between sparse and dense Gauss-Newton is again substantial.

\subsection{Cloth Control}
In our second inverse dynamics example we use the method by Geilinger et al.~\shortcite{geilinger2020add} to find time-varying handle positions for two corners of a sheet of cloth such that it moves from a given start configuration to a target state with prescribed positions. As best seen in the accompanying video, the optimized handle motion leads to two flip-overs, one in place and one with horizontal movement. As simulation model, we use a standard mass-spring system together with implicit Euler for time integration. To define the map between parameters and state for sensitivity analysis, we express the corresponding update rule in constraint form as
\begin{equation}
\bc^{t_i} = \bx^{t_i} - I_{IE}(\bx^{t_{i-1}}, \bp^{t_{i-1}}, \bp^{t_i} )
\end{equation}
where, given vertex positions $\bx^{t_{i-1}}$ as well as control forces $\bp^{t_{i-1}}$ and $\bp^{t_i} $, the implicit Euler rule $I_{IE}$ returns the new state $\bx^{t_i}$.
The cloth comprises 100 vertices and we perform $N$ simulation steps, leading to a total of $300N$ state and $6N$ control variables. We simulate for $1.66s$ of virtual time and set the step size accordingly. 
The goal of matching the target state is expressed as 
\begin{equation}
    f(\bx, \bp) = \frac12 \sum_{j \in \mathcal{S}} \norm{\bx^j - \tilde{\bx}^j }^2 
	+ R(\bp) \ ,
\end{equation}
where $\mathcal{S}$ is a set of keyframes and $R(\bp)$ is a regularizer that penalizes deviations from initial handle positions, handle velocities, and cloth velocities.

\begin{figure}
	\centering
	\raisebox{0.2\height}{\includegraphics[width=0.44\linewidth]{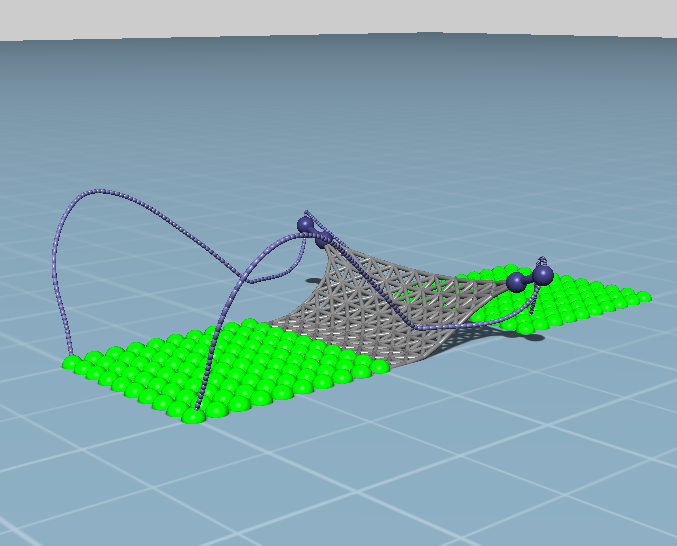}}
	\includegraphics[width=0.55\linewidth]{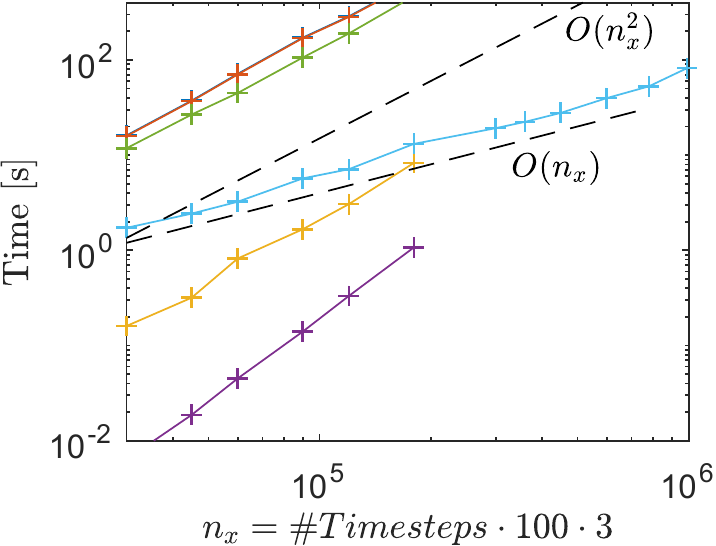}\\
	\vspace{0.3em}
	\hfill \includegraphics[width=0.65\linewidth]{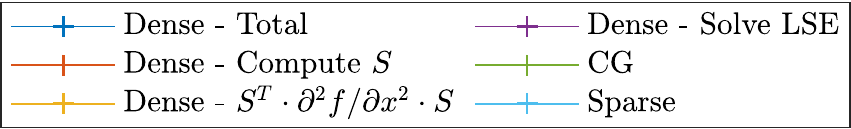}\\
	\caption{Time required to compute the search direction for the cloth control problem. \textit{Right}: we increase the number of time steps while keeping mesh resolution constant such that the problem size increases in terms of both state variables and parameters. The dense solution strategy ran out of memory for $n_x > 1.8 \cdot 10^{5}$. }
	\label{fig:cloth}
\end{figure}

For this example, we additionally provide break downs for the cost of DGN. Somewhat surprisingly, the factorization of the dense system only accounts for a fraction of the total time, which is dominated by the computation of the sensitivity matrix.
Fig. \ref{fig:cloth} shows that SGN outperforms DGN for all problem instances, whose size we control by the number of time steps used for forward simulation. 

As shown in Fig. \ref{fig:scalingCarCloth} (\textit{right}), when changing only the state size but keeping the number of parameters fixed, DGN scales better than SGN but breaks even only for very large problem sizes that  are intractable for dense solvers on current desktop machines. The difference in scaling between SGN and DGN can be explained by the fact that, unlike for the rod dome, the cost of solving the sparse system scales quadratically with state size which, in turn, can be attributed to the higher connectivity among state variables.

\section{Conclusions}

We presented a sparse Gauss-Newton solver for sensitivity analysis that eliminates the poor performance and scaling of the dense formulation. We have shown on a diverse set of examples that SGN scales asymptotically better than its dense counterpart in almost all cases. We have furthermore provided numerical evidence that SGN outperforms existing solvers for equilibrium-constrained optimization problems on many examples.

\subsection{Limitations \& Future Work}
All of our performance tests use a sparse direct solver, which imposes certain limits on the maximum problem size. One potential option for extending SGN to very large problems would be to use iterative saddle-point solvers such as the Uzawa algorithm. 

Sparse Gauss-Newton transforms a dense $n_p\times n_p$ system into a sparse system of dimension $(2n_x+n_p)\times(2n_x+n_p)$, where $n_x$ and $n_p$ denote the number of state and design variables, respectively. This transformation is only advantageous if the number of parameters is sufficiently large. For example, when optimizing for the Young's modulus of a homogeneous elastic solid, the dense $1\times1$ Hessian will always be faster to invert than its sparse counterpart. At the other extreme, SGN will generally be much faster when optimizing for per-element material coefficients. While the exact break-even point depends on the problem, our experiments show that already for small to moderate $n_p$, SGN outperforms DGN. 

Problems with sequential dependence between state variables (resulting, e.g., from time discretization) lead to a special block structure that can be leveraged to accelerate computation of the dense sensitivity matrix. We did not consider such problem-specific optimizations here.

Using a CG-based solver can be an attractive alternative, especially when fast back-substitutions are available. A disadvantage is that, for optimal performance, residual thresholds must be tweaked for each example. Furthermore, the convergence rate of CG depends strongly on the problem. Developing specialized pre-conditioners for  Eq.~\ref{eq:reducedGN} might be an interesting option for future work.

Our formulation assumes that the objective function can be expressed in nonlinear least squares form. While not all problems exhibit this particular form, they can often be reformulated or reasonably well approximated in this way.

Some of our examples include bound constraints on the parameters, which we enforced through log-barrier penalties or by simple projection of the search direction. The latter approach, however, is neither efficient nor guaranteed to converge in the general case. Incorporating bound and inequality constraints in our formulation is an interesting direction for future work.

We did not directly analyze the impact of the cost per simulation on optimization performance. In general, problems for which forward simulation is fast will benefit more from solvers that rely only on first-order derivative information but require more function evaluations. However, we believe that our selection of examples is representative for a large range of stiff inverse problems encountered in practice. For the case of non-stiff problems, on the other hand, inexact descent methods can be an attractive alternative \cite{Yan:2018:IDM}.

Finally, it would be interesting to extend our approach to efficiently compute second-order sensitivity information in the context of design space exploration for multi-objective optimization problems \cite{Schulz:2018:IED}.

\section*{Acknowledgements}
This research was supported by the Discovery Accelerator Awards program of the Natural Sciences and Engineering Research Council of Canada (NSERC) and
the European Research Council (ERC) under the European Union’s Horizon 2020 research and innovation program (Grant No. 866480). We thank the reviewers for their valuable feedback and suggestions.

\appendix

\section{Equivalence Result}
\label{sec:app}
We show that, for general objectives, using the dense system obtained for second-order sensitivity analysis and the sparse system (\ref{eq:sparseSystemGN}) lead to the same search direction. This proof also shows the equivalence between sensitivity analysis and sequential quadratic programming for the special case of equality constraints that are enforced to stay satisfied at all times. 

\begin{theorem}\label{th:sparseSolve}
Let $A, B, C, \frac{\partial c}{\partial x}, \frac{\partial c}{\partial p}$ denote sparse matrices with the correct dimensions and let $H = \frac{d \bx}{d \bp} A \frac{d \bx}{d \bp}^T + B \frac{d \bx}{d \bp}+ \frac{d \bx}{d \bp}^T B^T + C$. To compute the solution to the dense linear system $H \cdot \delta \bp = - \frac{d f}{d \bp}^T$, we can equivalently solve the larger sparse system (\ref{eq:sparseSystemGN}).
\end{theorem}
\begin{proof}
See construction in Sec. \ref{sec:SGN}.
\end{proof}
We next introduce the Lagrangian for the optimization problem (\ref{eq:optProb}) as
\begin{equation}
\Lagr (\bx, \bp, \blambda) = f(\bx, \bp) + \blambda^T \bc(\bx, \bp)
\end{equation}
whose gradient is 
\begin{equation}
\nabla \Lagr (\bx, \bp, \blambda) =
\begin{bmatrix}
\nabla_\bx f + \nabla_\bx \bc ^T  \cdot  \blambda \\
\nabla_\bp f + \nabla_\bp \bc ^T  \cdot  \blambda \\
\bc
\end{bmatrix} \ ,
\end{equation}
where $\blambda$ are the Lagrange multipliers. The first-order optimality (or KKT) conditions correspond to $\nabla \Lagr (\bx, \bp, \blambda)=\mathbf{0}$.
Solving these conditions with Newton's method leads to the so-called KKT system
\begin{equation*}
\label{eq:KKTsystem}
\begin{bmatrix}
\nabla^2_{\bx\bx} f + \nabla^2_{\bx\bx}\bc :\blambda & \nabla^2_{\bx\bp} f + \nabla^2_{\bx\bp}\bc :\blambda  & \nabla_\bx \bc ^T\\
\nabla^2_{\bp\bx} f + \nabla^2_{\bp\bx}\bc :\blambda & \nabla^2_{\bp\bp} f + \nabla^2_{\bp\bp}\bc :\blambda  & \nabla_\bp \bc ^T\\
\nabla_\bx \bc & \nabla_\bp \bc & 0
\end{bmatrix} 
\begin{bmatrix}
    \delta \bx \\
    \delta \bp \\
    \delta \blambda
\end{bmatrix}
 = -
 \begin{bmatrix}
    \nabla_\bx \Lagr\\
    \nabla_\bp \Lagr \\
    \bc
 \end{bmatrix}
\ ,
\end{equation*}
where we used the shorthand $\nabla^2_{\by\mathbf{z}} \bc:\blambda=\sum_i \blambda_i \nabla^2_{\by\mathbf{z}} \bc_i$.
\begin{theorem}
When using the \textit{adjoint} variables as Lagrange multipliers
\begin{equation}
\lambdaNatural = - \left[ \frac{\partial \bc}{\partial \bx} \right]^{-T} \frac{\partial f}{\partial \bx}^T \ ,
\end{equation}
the KKT system and the system for second-order sensitivity analysis, $\frac{d f^2}{d\bp^2}\delta \bp = -\frac{d f}{d\bp}^T$, give the same search direction $\delta \bp$.
\end{theorem}
\begin{proof}
We show that, for $\blambda = \lambdaNatural$, we have 
\begin{align}
\frac{d f^2}{d\bp^2} &= 
    \frac{d \bx^T}{d \bp} \frac{\partial^2 f}{ \partial \bx^2} \frac{d \bx}{d \bp} 
    + \frac{\partial^2 f}{ \partial \bx \partial \bp} \frac{d \bx}{d \bp}
    +  \frac{d \bx^T}{d \bp} \frac{\partial^2 f}{\partial \bp \partial \bx}
    +\frac{\partial^2 f}{\partial \bp^2}
    + \sum_i \frac{\partial f}{\partial \bx_i} \frac{d^2 \bx_i}{d \bp^2} 
        \nonumber \\
    &=
    \frac{d \bx^T}{d \bp} \nabla^2_{\bx\bx} \Lagr \frac{d \bx}{d \bp}
    + \nabla^2_{\bx\bp} \Lagr \frac{d \bx}{d \bp}
    + \frac{d \bx^T}{d \bp} \nabla^2_{\bp\bx} \Lagr
    +  \nabla^2_{\bp\bp} \Lagr
    \label{eq:KKTReducedEqualToTotalHessian}
\end{align} 
Only the term involving second-order sensitivities is non-obvious. Using basic transformations, we obtain
{\tiny
\begin{align}
    \sum_k &\frac{\partial f}{\partial \bx_k} \frac{d^2 \bx_k}{d \bp_j d \bp_i} =\\
     &= - \sum_k \frac{\partial f}{\partial \bx_k} \left[\left(\frac{\partial \bc}{\partial \bx} \right)^{-1} \left(
    \frac{d \bx_m}{d \bp_j} \frac{\partial^2 \bc}{\partial \bx_m \partial \bx_l} \frac{d \bx_l}{d \bp_i}
    + \frac{\partial^2 \bc}{\partial \bx_m \partial \bp_i} \frac{d \bx_m}{d \bp_j}
    + \frac{d \bx_l}{d \bp_i} \frac{\partial^2 \bc}{\partial \bp_j \partial \bx_l} 
    +
     \frac{\partial^2 \bc}{\partial \bp_j \partial \bp_i}
     \right) \right]_k
    \nonumber\\
    &= -\sum_k \left( \left(\frac{\partial \bc}{\partial \bx} \right)^{-T} \frac{\partial f}{\partial \bx}^T \right)_k \left(
    \frac{d \bx_m}{d \bp_j} \frac{\partial^2 \bc_k}{\partial \bx_m \partial \bx_l} \frac{d \bx_l}{d \bp_i}
    + \frac{\partial^2 \bc_k}{\partial \bx_m \partial \bp_i} \frac{d \bx_m}{d \bp_j}
    + \frac{d \bx_l}{d \bp_i} \frac{\partial^2 \bc_k}{\partial \bp_j \partial \bx_l} 
    + \frac{\partial^2 \bc_k}{\partial \bp_j \partial \bp_i} \right)
    \nonumber\\
    &= \sum_k \lambdaNatural_k \left(
    \frac{d \bx_m}{d \bp_j} \frac{\partial^2 \bc_k}{d \bx_m \partial \bx_l} \frac{d \bx_l}{d \bp_i}
    + \frac{\partial^2 \bc_k}{\partial \bx_m \partial \bp_i} \frac{d \bx_m}{d \bp_j}
    + \frac{d \bx_l}{d \bp_i} \frac{\partial^2 \bc_k}{\partial \bp_j \partial \bx_l} 
    + \frac{\partial^2 \bc_k}{\partial \bp_j \partial \bp_i} \right) \ .
\end{align}
}
Eq. \eqref{eq:KKTReducedEqualToTotalHessian} thus holds and, using Theorem \ref{th:sparseSolve}, the result follows directly.
\end{proof}

\section{Block Solve}
\label{app:blockSolve}
Using $B=0$ and $C=0$ in (\ref{eq:sparseSystemGN}), we have
\begin{equation*}
	\delta \blambda = -\frac{\partial \bc}{\partial \bp}^{-T}\frac{df}{d\bp}^T \ , \quad \text{and} \quad
	\delta \bx = -A^{-1}\frac{\partial \bc}{\partial \bx}^{T}\delta \blambda \ .
\end{equation*}
Using these definitions in the third row of (\ref{eq:sparseSystemGN}), we obtain
\begin{align*}
	\delta \bp & = -\frac{\partial \bc}{\partial \bp}^{-1}\frac{\partial \bc}{\partial \bx}\delta \bx 
				= \frac{\partial \bc}{\partial \bp}^{-1}\frac{\partial \bc}{\partial \bx} A^{-1}\frac{\partial \bc}{\partial \bx}^{T}\delta \blambda  \\
				& = -\frac{\partial \bc}{\partial \bp}^{-1}\frac{\partial \bc}{\partial \bx} A^{-1}\frac{\partial \bc}{\partial \bx}^{T}\frac{\partial \bc}{\partial \bp}^{-T}\frac{df}{d\bp}^T \ , 
\end{align*}
and using (\ref{eq:totalGradient}) with $\frac{\partial f}{\partial \bp}=\mathbf0$, we finally have
\begin{align*}
	\delta \bp  & =   \frac{\partial \bc}{\partial \bp}^{-1}\frac{\partial \bc}{\partial \bx} A^{-1}\frac{\partial \bc}{\partial \bx}^{T}\frac{\partial \bc}{\partial \bp}^{-T}\left(\frac{\partial \bc}{\partial \bp}^T\frac{\partial \bc}{\partial \bx}^{-T} \right )\frac{\partial f}{\partial \bx}^T \\ 
	& = \frac{\partial \bc}{\partial \bp}^{-1}\frac{\partial \bc}{\partial \bx}A^{-1}\frac{\partial f}{\partial \bx}^T \ .
\end{align*}

\bibliographystyle{ACM-Reference-Format}
\bibliography{references}

\clearpage
\newpage
\includepdf[pages=-]{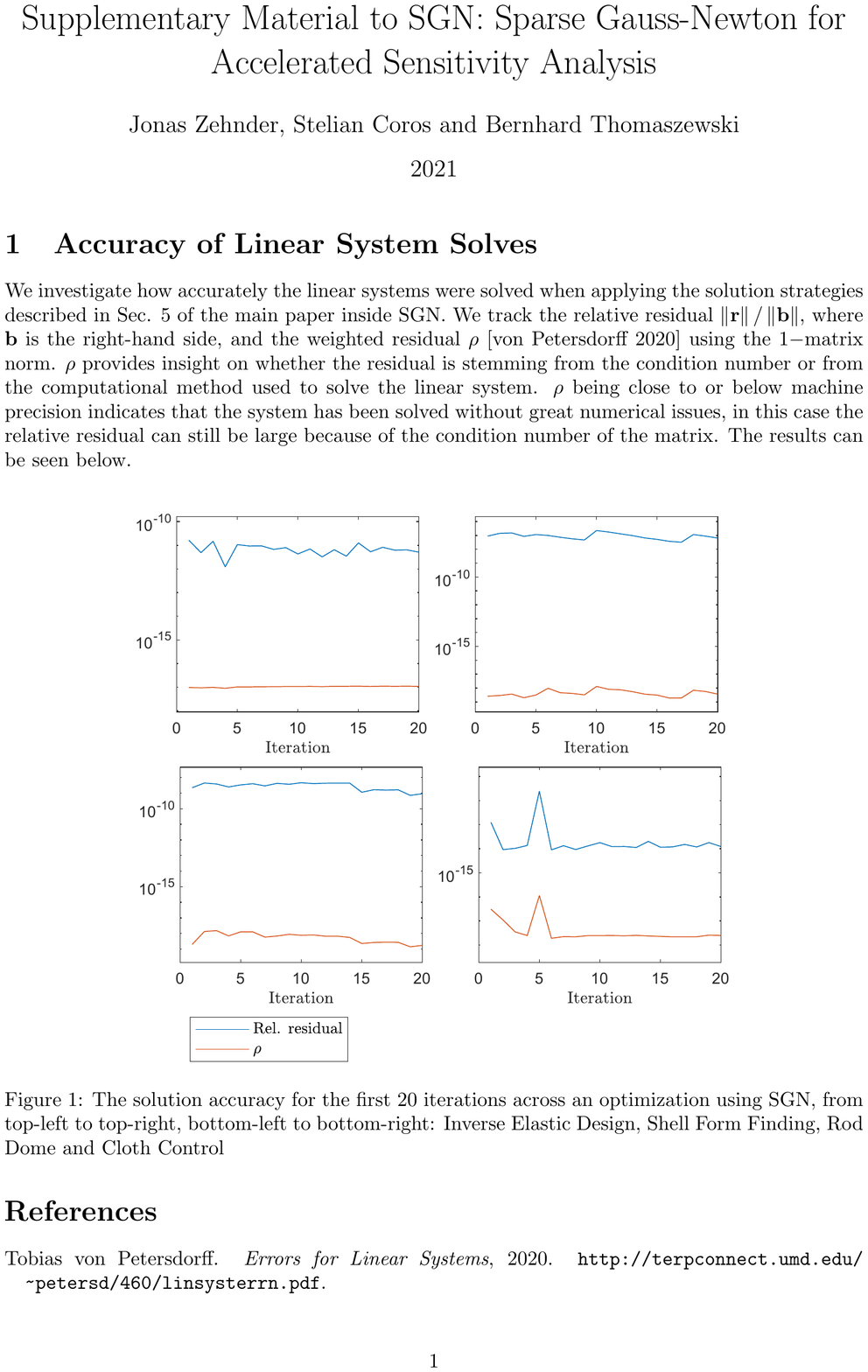}

\end{document}

\maketitle
\vspace{-2em}
\section{Accuracy of Linear System Solves}

We investigate how accurately the linear systems were solved when applying the solution strategies described in Sec. 5 of the main paper inside SGN. We track the relative residual $\norm{\br}/\norm{\bb}$, where $\bb$ is the right-hand side, and the weighted residual $\rho$ \citep{ErrorsInLinearSystems} using the $1-$matrix norm. $\rho$ provides insight on whether the residual is stemming from the condition number or from the computational method used to solve the linear system. $\rho$ being close to or below machine precision indicates that the system has been solved without great numerical issues, in this case the relative residual can still be large because of the condition number of the matrix. The results can be seen below.

\vspace{1em}

\begin{figure}[H]
    \centering
    \includegraphics[width=0.34\linewidth]{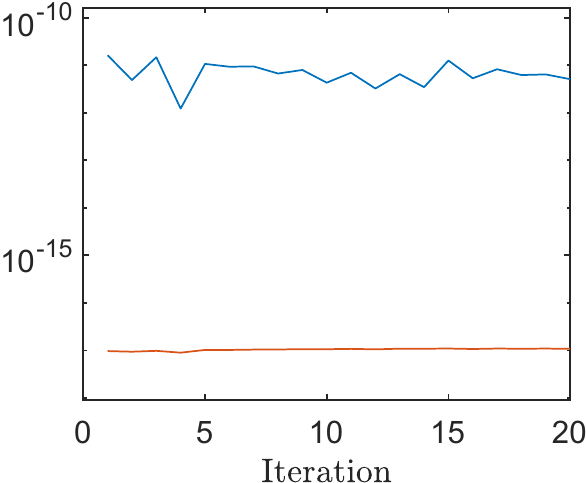}
    \includegraphics[width=0.34\linewidth]{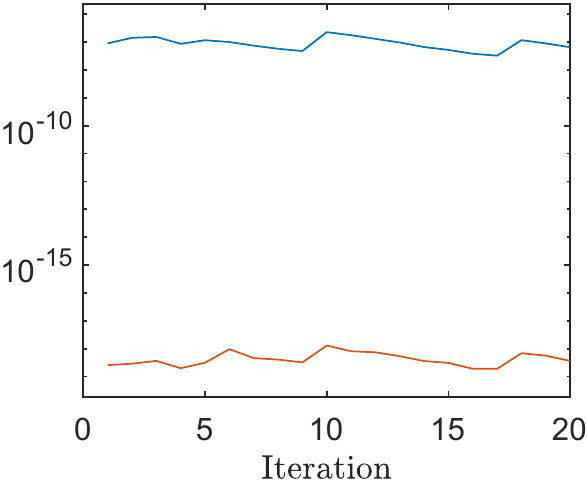}
    \vspace{0.5em}
    \\
    \includegraphics[width=0.34\linewidth]{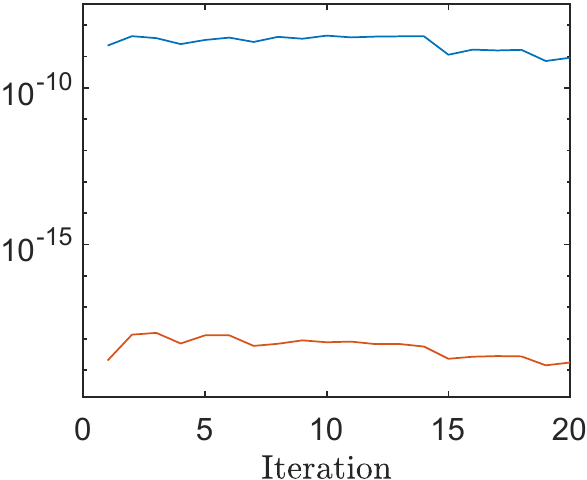}
    \includegraphics[width=0.34\linewidth]{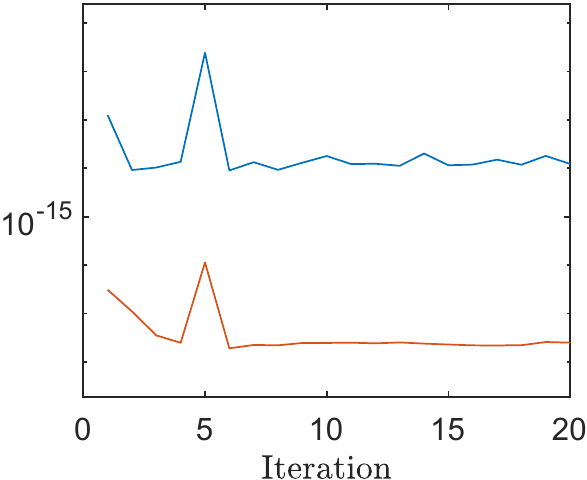}\\ \vspace{0.5em}
    \hspace{-16em}\includegraphics[width=2.8cm]{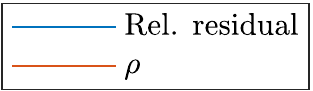}
    \caption{The solution accuracy for the first 20 iterations across an optimization using SGN, from top-left to top-right, bottom-left to bottom-right: Inverse Elastic Design, Shell Form Finding, Rod Dome and Cloth Control}
    \label{fig:linearSolvercomp}
\end{figure}

%\vspace{0.2em}

\bibliographystyle{plainnat}
\bibliography{references}